\documentclass[a4paper,12pt]{article}
\usepackage[utf8]{inputenc}

\usepackage{mathrsfs}
\usepackage{graphicx}
\usepackage{enumerate}
\usepackage{multicol}
\usepackage{color}
\usepackage{amsmath,amssymb,amscd}
\usepackage{amsthm}
\usepackage{mathabx}

\usepackage{pdfpages}
\usepackage{hyperref}

\usepackage{times}
\usepackage[varg]{txfonts}

\usepackage[top=1in, bottom=1in, left=0.75in, right=0.75in]{geometry}

\theoremstyle{plain}
\newtheorem{thm}{Theorem}

\theoremstyle{definition}

\newtheorem{ex}{Example}

\theoremstyle{remark}

\newtheorem*{ackn}{Acknowledgment}

\newcommand{\C}{\mathbb{C}}
\newcommand{\R}{\mathbb{R}}
\newcommand{\N}{\mathbb{N}}

\title{Moment functions of higher rank on polynomial hypergroups}
\author{{\.Z}ywilla Fechner, Eszter Gselmann and L\'{a}szl\'{o} Sz\'ekelyhidi}

\begin{document}
\nocite{*}

\maketitle

\begin{abstract}
In this paper we consider generalized moment functions of higher order. These functions are closely related to the well-known functions of binomial type which have been investigated on various abstract structures. In our former paper we investigated the properties of generalized moment functions of higher order on commutative groups. In particular, we proved the characterization of generalized moment functions on a commutative group as the product of an exponential and composition of multivariate Bell polynomial and a sequence additive functions. In the present paper we continue the study of generalized moment function sequences of higher order in the more abstract setting, namely we consider functions defined on a hypergroup. We characterize these functions on the polynomial hypergroup in one variable by means of partial derivatives of a composition of polynomials generating the polynomial hypergroup and an analytic  function. As an example, we give an explicit formula for moment generating functions of rank at most two on the Tchebyshev hypergroup.
\end{abstract}

\section{Introduction}\label{intro}

Let $G$ be a locally  compact  Abelian topological group. Recall that a nonzero continuous function $m\colon G\to \C$ is called {\it exponential}, if 
$$
m(x+y)=m(x)m(y)
$$ 
holds for all $x,y$ in $G$. 
\vskip.2cm

Exponentials can be considered as "generalized moment functions" of order zero. Indeed, a sequence $\varphi_n:G\to \C$ of continuous functions is called a {\it generalized moment function sequence}, if $\varphi_0(0)=1$ and
\begin{equation}
\varphi_n(x+y)=\sum_{k=0}^n \binom{n}{k} \varphi_k(x)\varphi_{n-k}(y)
\label{mom1}
\end{equation} 
holds for all $x,y$ in $G$ and $n$ in $\N$. 
In this case $\varphi_0$ is an arbitrary exponential function and  the function $\varphi_k$ is a {\it generalized moment function of order} $k$. 
\vskip.2cm

Equation \eqref{mom1} is closely related to the well-known functions of binomial type. 
A detailed discussion about binomial type equations in abstract setting, which has been the motivation for the present research, can be found in \cite{MR440237}, where it was shown that if $G$ is a grupoid and $R$ is a commutative ring, then functions $\varphi_n\colon G\to R$ satisfying \eqref{mom1} for each $n$ in $\N$ are of the form
\begin{equation}
\varphi_n(t)=n!\sum_{j_1+2j_2+\dots+nj_n=n}\prod_{k=1}^n \frac{1}{j_k!}\left(\frac{a_k(t)}{k!}  \right)^{j_k}
\label{eq:AczelPolySol}
\end{equation}
for all $t$ in $G$ and $k$ in $\N$ with arbitrary homomorphisms $a_k$ from $G$ into $R$.
\vskip.2cm

The system of functional equations \eqref{mom1}, as well as the concept of moment functions can be and has been generalized for commutative hypergroups. Formally, the group operation $+$ in \eqref{mom1} is replaced by $*$, the convolution defined on the commutative hypergroup $X$, and we obtain the system of equations
\begin{equation}
	\varphi_n(x*y)=\sum_{k=0}^n \binom{n}{k} \varphi_k(x)\varphi_{n-k}(y),
	\label{mom2}
\end{equation} 
where $x,y$ are in $X$, and $n=0,1,2,\dots$. We also assume that $\varphi_0(o)=1$, where $o$ is the identity of the hypergroup $X$.  We recall that in the hypergroup setting equation \eqref{mom2} is a system of integral equations of the following form:
$$
\int_X \varphi_n(t)\,d(\delta_x*\delta_y)(t)=\sum_{k=0}^n \binom{n}{k} \varphi_k(x)\varphi_{n-k}(y).
$$
For more details see \cite{BloHey95} or \cite{MR2978690}. The functions $\varphi_n$ are supposed to be continuous.

\section{Moment functions of higher rank}

Further generalization is available in the following way. 
Let $X$ be a commutative hypergroup, $r$ a positive integer, and for each multi-index $\alpha$ in $\N^r$ let $\varphi_{\alpha}:X\to \C$ be a continuous function. We say that $(\varphi_{\alpha})_{\alpha \in \mathbb{N}^{r}}$ is a \emph{generalized 
		moment sequence of rank $r$}, if 
	\begin{equation}\label{Eq3}
		\varphi_{\alpha}(x*y)=\sum_{\beta\leq \alpha} \binom{\alpha}{\beta} \varphi_{\beta}(x)\varphi_{\alpha-\beta}(y)
	\end{equation}
holds whenever $x,y$ are in $X$. We may consider finite sequences as well, if we restrict $|\alpha|\leq N$ with some nonnegative integer $N$.

For simplicity we use the term {\it moment sequence} omitting the adjective "generalized". We call a continuous function $\varphi:X\to \C$ a {\it moment function}, if there is a positive integer $r$, a moment sequence $(\varphi_{\alpha})_{\alpha \in \mathbb{N}^{r}}$  of rank $r$, and a multi-index $\alpha$ in $\N^r$ such that $\varphi=\varphi_{\alpha}$. We recall that, using multi-indeces, besides the usual vector-notation for the basic operations we use the following notation: for $\alpha=(\alpha_1,\alpha_2,\dots,\alpha_r)$ and $\beta=(\beta_1,\beta_2,\dots,\beta_r)$ in $\N^r$ we shall write
$\alpha\leq \beta\enskip\text{whenever}\enskip \alpha_i\leq \beta_i\enskip\text{for} \enskip i=1,2,\dots r$, and $\alpha<\beta\enskip\text{whenever}\enskip \alpha\leq\beta\enskip\text{and}\enskip \alpha\ne \beta$. Further, we use the notations
$$
|\alpha|=\alpha_1+\alpha_2+\cdots+\alpha_r,\hskip1cm \alpha!=\alpha_1!\cdot \alpha_2!\cdots\alpha_r!,
$$
$$
\binom{\alpha}{\beta}=\frac{\alpha!}{\beta! \cdot (\alpha-\beta)!},\hskip1cm x^{\alpha}=x_1^{\alpha_1}\cdot x_2^{\alpha_2}\cdot\dots\cdot x_r^{\alpha_r}.
$$
If there is no misunderstanding, the zero of $\N^r$ will be denoted by $0$ instead of $(0,0,\dots,0)$. 

In our former paper \cite{FecGseSze20} we have investigated moment function sequences of higher rank on groups. For their description we used Bell polynomials (see the definition in \cite{FecGseSze20}). We proved the following result (see Proposition 3. and Theorem 2. in \cite{FecGseSze20}):

\begin{thm}
Let $G$ be a commutative group, $r$ a positive integer, and for each $\alpha$ multi-index in $\N^r$. The functions $f_{\alpha} : G\mapsto \C$  form a generalized moment sequence of rank $r$ if and only if there exists an exponential \hbox{$m:G\mapsto \C$} and a family of complex-valued additive functions $a=(a_{\alpha})$ such that for every multi-index $\alpha$ in $\N^r$ and $x$ in $G$ we have
	$$
	f_{\alpha}(x)=B_{\alpha}\big(a_{\alpha}(x)\big)m(x).
	$$
\end{thm}

It is reasonable to ask if a similar description of moment functions of higher rank is available in the hypergroup case.  The answer is negative: due to the meaning of the symbol $f(x*y)$ as an integral products of functions satisfying simple functional equations will not preserve their  properties arising from these equations. For instance, products of exponentials on hypergroups is, in general, not an exponential. In the subsequent  paragraphs we will show that still the idea of multiplying exponentials by Bell polynomials can be replaced by application of  some appropriate differential operators.

\section{Main results}

We begin with recalling the definition of a polynomial hypergroup in one variable. Let $(a_n)_{n\in \N}$, $(b_n)_{n\in \N}$ and $(c_n)_{n\in \N}$ be real sequences with the following properties
$$ c_n>0,\quad b_n\geq 0, \quad a_{n+1}>0 $$
for each $n$ in $\N$. Moreover, $a_0=b_0=0$ and 
$$ 
a_n+b_n+c_n=1
$$
for each $n$ in $\N$. We define the sequence of polynomials $(P_n)_{n\in\N}$ by the formulas $P_0(\lambda)=1$, $P_1(\lambda)=\lambda$ and

$$\lambda P_n(\lambda)=a_nP_{n-1}(\lambda)+b_nP_n(\lambda)+c_nP_{n+1}(\lambda)$$
for each $n\geq 1$ and $\lambda$ in $\R$.
One can show that for each $k,m,n$ in $\N$ there exist constnts $c(n,m,k)$ such that
\begin{equation}\label{eq:LinFormula}
P_n\cdot P_m=\sum_{k=|n-m|}^{n+m}c(n,m,k)P_k.
\end{equation}
for each $m,n$ in $\N$.

Formula \eqref{eq:LinFormula} is called linearization formula. One can show that

$$
\sum_{k=|n-m|}^{n+m}c(n,m,k)=1
$$ 
for each $m,n$ in $\N$. 

If $c(n,m,k)\geq 0$ for all $k,m,n$ in $\N$, then on the set $\N$ we can define the hypergroup structure, where the convolution is given by
$$\delta_n*\delta_m= \sum_{k=|n-m|}^{n+m}c(n,m,k)\delta_k.$$

Assume now that $X$ is a polynomial hypergroup generated by the sequence of polynomials $(P_n)_{n\in\N}$ and $r,N$ are positive integers. Our purpose is to describe the general solution of the system of functional equations \eqref{Eq3} for the unknown functions $\varphi_{\alpha}:\N\to\C$ with $|\alpha|\leq N$.
\vskip.2cm

In the case $r=1$ the system \eqref{Eq3} reduces to the system \eqref{mom1} which was solved in \cite{MR2161803} (see also \cite[Theorem 2.5, p. 44]{MR2978690}). We recall the result.

\begin{thm}\label{r1}
Let $X$ be a polynomial hypergroup generated by the sequence of polynomials $(P_n)_{n\in \N}$ and let $N$ be a positive integer. The sequence of functions $\varphi_{k}:\N\to\C$ $(k=0,1,\dots, N)$ is a moment function sequence (of rank $1$) if and only if they have the form
\begin{equation}\label{genformr1} 
	\varphi_k(n)=(P_n\circ f)^{(k)}(0)\hskip .5cm k=0,1,\dots,N; n\in \N,
\end{equation}
where
\begin{equation}\label{r1f}
f(t)=\sum_{j=0}^N \frac{\varphi_j(1)}{j!}t^j\enskip\text{for}\enskip t \enskip\text{in}\enskip \R. 
\end{equation}
\end{thm}

The latter theorem says that moment functions can be reperesented as 
\[
 \varphi_{\alpha}(n)= \partial^{\alpha}(P_{n}\circ f)(t)\vert_{t=0}. 
\]
The derivative of the above composition can be computed with the aid of the multivariate Faá di Bruno formula, and using this, we get that 
\[
 \partial^{\alpha}(P_{n}\circ f)(t)\vert_{t=0}= 
 \sum_{\beta \leq \alpha} \partial^{\beta}P_{n}(f(t))\cdot B_{\alpha, \beta}(f(t), \ldots, \partial^{\alpha}f(t))\vert_{t=0}
\]
If we use that 
\[
 \partial^{\alpha}f(t)\vert_{t=0}= \varphi_{\alpha}(1), 
\]
then the above formula becomes simpler. For more details see the paper of A.~Schumann on arXiv: \href{https://arxiv.org/pdf/1903.03899.pdf}{https://arxiv.org/pdf/1903.03899.pdf}.

Our next theorem generalizes this result for moment function sequences of higher rank.

\begin{thm}\label{r}
	Let $X$ be a polynomial hypergroup generated by the sequence of polynomials $(P_n)_{n\in \N}$ and let $r,N$ be positive integers. The sequence of functions $\varphi_{\alpha}:\N\to\C$ $(\alpha\in \N^r, |\alpha|\leq N)$ is a moment function sequence of rank $r$ if and only if they have the form
	\begin{equation}\label{genformr} 
		\varphi_{\alpha}(n)=\partial^{\alpha}(P_n\circ f)(0)\hskip .5cm k=0,1,\dots,N; n\in \N,
	\end{equation}
	where $f:\R^r\to\C$ is defined for $t$ in $\R^r$ by
	\begin{equation}\label{rf}
		f(t)=\sum_{|\alpha|\leq N} \frac{\varphi_{\alpha}(1)}{\alpha!}t^{\alpha}\enskip\text{for}\enskip t \enskip\text{in}\enskip \R^r. 
	\end{equation}
\end{thm}

\begin{proof}
We show the sufficiency first. We start with the identity
$$
P_m(\lambda)P_n(\lambda)=P_{m*n}(\lambda),
$$
where 
$$
P_{m*n}(\lambda)=\sum_{k=|m-n|}^{m+n} c(m,n,k)P_k(\lambda)
$$
and $\lambda$ is an arbitrary complex number. We substitute $\lambda=f(t)$, where $t$ is arbitrary in $\R^r$:
\begin{equation}\label{Eq1}
P_m(f(t))P_n(f(t))=P_{m*n}(f(t)),
\end{equation}
and apply $\partial^{\alpha}$ on both sides to get

\begin{equation}\label{Eq2}
\partial^{\alpha}(P_{m*n}(f(t))=\sum_{\beta\leq \alpha} \binom{\alpha}{\beta} \partial^{\beta}P_m(f(t))\cdot \partial^{\alpha-\beta}P_n(f(t)).
\end{equation}

Now we substitute $t=0$ to obtain
$$
\varphi_{\alpha}(m*n)=\partial^{\alpha}(P_{m*n}\circ f)(0)
$$
$$
=\sum_{\beta\leq \alpha}\binom{\alpha}{\beta}  \partial^{\beta}(P_{m}\circ f)(0)\cdot \partial^{\alpha-\beta}(P_{n}\circ f)(0)=\sum_{\beta\leq \alpha}\binom{\alpha}{\beta} \varphi_{\beta}(m)\cdot \varphi_{\alpha-\beta}(n),
$$
that is, the functions $\varphi_{\alpha}$ for $|\alpha|\leq N$ form a moment sequence of rank $r$, which proves the sufficiency part of the theorem.
\medskip

Now suppose that the sequence of functions $\varphi_{\alpha}:\N\to\C$ for $\alpha$ in  $\N^r$ with $|\alpha|\leq N$ is a moment function sequence of rank $r$ on $X$. We define the function $f$ as given in \eqref{rf} and the functions $\psi_{\alpha}:\N\to\C$ by
\begin{equation}\label{psi}
	\psi_{\alpha}(n)=\varphi_{\alpha}(n)-\partial^{\alpha}(P_n\circ f)(0)
\end{equation}
for $\alpha$ in $\N^r $ with $|\alpha|\leq N$ and for each $n=0,1,\dots$. We show that all functions $\psi_{\alpha}$ vanish identically. For $\alpha=0$ we have
$$
\psi_0(n)=\varphi_0(n)-P_n(\varphi_0(1))
$$
whenever $n=0,1,\dots$. As $\varphi_0$ is an exponential, by Theorem 2.2 in \cite{MR2978690}, it follows $\varphi_0(n)=P_n(\varphi_0(1))$, hence $\psi_0=0$.
\vskip.2cm

Now we show by induction on $|\alpha|$ that $\varphi_{\alpha}(0)=0$ whenever $|\alpha|>0$. Indeed, $\varphi_{\alpha}$ is a $\varphi_0$-sine function for each $\alpha$ with $|\alpha|=1$, hence $\varphi_{\alpha}(0)=0$. Now let $|\alpha|>1$, and assume that we have proved that $\varphi_{\beta}(0)=0$ for each $\beta$ with $0<|\beta|<|\alpha$. Then we substitute $x=y=0$ in \eqref{Eq3} to obtain
$$
\varphi_{\alpha}(0)=\varphi_{\alpha}(0)+\sum_{0<\beta<\alpha} \binom{\alpha}{\beta} \varphi_{\beta}(0)\varphi_{\alpha-\beta}(0)+\varphi_{\alpha}(0),
$$
that is
$$
\varphi_{\alpha}(0)=2\varphi_{\alpha}(0),
$$
which implies $\varphi_{\alpha}(0)=0$. This implies $\psi_{\alpha}(0)=0$ whenever $|\alpha|>0$. On the other hand, for $|\alpha|>0$ we have
$$
\psi_{\alpha}(1)=\varphi_{\alpha}(1)-\partial^{\alpha}(P_1\circ f)(0)=\varphi_{\alpha}(1)-\partial^{\alpha}f(0)=\varphi_{\alpha}(1)-\varphi_{\alpha}(1)=0.
$$

We show, by induction on $|\alpha|$, that $\psi_{\alpha}=0$ for each $\alpha$. This clearly holds for $\alpha=0$. Assuming that $\psi_{\beta}=0$ for $\beta<\alpha$ we
apply the linearization formula \eqref{Eq2} with $m=1$ and $t=0$ to get
\begin{equation*}
\partial^{\alpha}(P_{n*1}\circ f)(0)=\sum_{\beta\leq \alpha} \binom{\alpha}{\beta} \partial^{\alpha-\beta}(P_1\circ f)(0)\cdot \partial^{\beta}(P_n\circ f)(0),
\end{equation*}
or for $n\geq 1$
\begin{equation*}
	\sum_{l=n-1}^{n+1}c(n,1,l)\partial^{\alpha}(P_l\circ f)(0)=\sum_{\beta\leq \alpha} \binom{\alpha}{\beta} \varphi_{\alpha-\beta}(1)\cdot \partial^{\beta}(P_n\circ f)(0),
\end{equation*}
that is
\begin{equation}\label{Eq4}
	\sum_{\beta< \alpha} \binom{\alpha}{\beta} \varphi_{\alpha-\beta}(1)\cdot \partial^{\beta}(P_n\circ f)(0)=
\end{equation}	
$$	
	\sum_{l=n-1}^{n+1}c(n,1,l)\partial^{\alpha}(P_l\circ f)(0)-\varphi_0(1)  \partial^{\alpha}(P_n\circ f)(0).
$$
On the other hand, by the definition of the moment function sequence we have
$$
\sum_{\beta\leq \alpha} \binom{\alpha}{\beta} \varphi_{\beta}(n)\varphi_{\alpha-\beta}(1)=\varphi_{\alpha}(n*1),
$$
which can be rewritten as
\begin{equation*}
\sum_{\beta<\alpha} \binom{\alpha}{\beta}\varphi_{\alpha-\beta}(1)\cdot \varphi_{\beta}(n)=\varphi_{\alpha}(n*1)-\varphi_0(1)\varphi_{\alpha}(n),
\end{equation*}
or
\begin{equation}\label{Eq5}
	\sum_{\beta<\alpha} \binom{\alpha}{\beta}\varphi_{\alpha-\beta}(1)\cdot \varphi_{\beta}(n)=
\end{equation}	
$$	
	\sum_{l=n-1}^{n+1}c(n,1,l)\varphi_{\alpha}(l)-\varphi_0(1)\varphi_{\alpha}(n).
$$
We subtract \eqref{Eq4} from \eqref{Eq5}, then, by assumption, we have 
$$
	\sum_{l=n-1}^{n+1}c(n,1,l)\psi_{\alpha}(l)=\varphi_0(1)\psi_{\alpha}(n).
$$
This is a second order linear recursion for the sequence $n\mapsto \psi_{\alpha}(n)$ with  $\psi_{\alpha}(0)=\psi_{\alpha}(1)=0$, which implies $\psi_{\alpha}=0$. This holds for each $\alpha$ with $|\alpha|\leq N$, hence our theorem is proved.
\end{proof}

As an example we calculate the generalized moment functions of rank at most $2$ in the case $r=2$ on the Tchebyshev hypergroup. 
\begin{ex}
The generating functions are the Tchebyshev polynomials of the first kind:
$$
T_n(\lambda)=\cos (n \arccos \lambda)\hskip.2cm \text{for}\hskip.2cm n=0,1,\dots
$$

By formula \eqref{genformr} we have
$$
\varphi_{\alpha}(n)=\partial^{\alpha} (T_n\circ f)(0,0)
$$
whenever $\alpha=(0,0), (1,0), (0,1), (1,1), (2,0), (0,2)$, and $f$ is defined by
equation \eqref{rf}. We denote $\varphi_{0,0}(1)=T_1(\varphi_{0,0}(1))$ by $\lambda$, then we have

\begin{eqnarray*}
	\varphi_{0,0}(n)&=&T_n(\lambda)\\
	\varphi_{1,0}(n)&=&c_{1,0}T_n'(\lambda), \hskip.3cm \\
	\varphi_{0,1}(n)&=&c_{0,1}T_n'(\lambda), \hskip.3cm \\ 
	\varphi_{1,1}(n)&=&c_{1,0}c_{0,1}T_n''(\lambda)+c_{1,1}T_n'(\lambda)\\
	\varphi_{2,0}(n)&=&c_{1,0}^2T_n''(\lambda)+c_{2,0}T_n'(\lambda), \hskip.3cm \\
	\varphi_{0,2}(n)&=&c_{0,1}^2T_n''(\lambda)+c_{0,2}T_n'(\lambda).
\end{eqnarray*}
\end{ex}

\begin{ackn}
 The research of E.~Gselmann has partially been carried out with the help of the project 2019-2.1.11-T\'{E}T-2019-00049,
which has been implemented with the support provided from NRDI (National Research, Development
and Innovation Fund of Hungary), financed under the T\'{E}T funding scheme.
\\
The research of E.~Gselmann and L.~Sz\'{e}kelyhidi has been supported by the NRDI (National Research, Development
and Innovation Fund of Hungary) Grant no. K 134191.
\end{ackn}

\end{document}